\documentclass[twoside,bysec]{matstud}
\usepackage{amssymb,amscd}
\newtheorem{theorem}{Theorem}[section]
\newtheorem{proposition}[theorem]{Proposition}
\newtheorem{lemma}[theorem]{Lemma}
\newtheorem{corollary}[theorem]{Corollary}
\theoremstyle{definition}
\newtheorem{remark}[theorem]{Remark}
\newtheorem{question}[theorem]{Question}
\newtheorem{definition}[theorem]{Definition}
\newtheorem{example}[theorem]{Example}

\newcommand{\HP}{\hat P}
\newcommand{\Pt}{P_\tau}
\newcommand{\M}{{\mathcal M}}
\newcommand{\A}{{\mathcal A}}
\def\P{{\mathcal P}}
\newcommand{\F}{{\mathcal F}}
\newcommand{\G}{{\mathcal G}}
\newcommand{\N}{{\mathcal N}}
\newcommand{\IN}{{\Bbb N}}
\newcommand{\IR}{{\Bbb R}}
\newcommand{\Comp}{{\mathcal C}omp}
\newcommand{\Tych}{{\mathcal T}ych}

\newcommand{\supp}{\operatorname{supp}}

\newcommand{\id}{\operatorname{id}}
\newcommand{\pr}{\operatorname{pr}}
\newcommand{\oplu}{\operatornamewithlimits{\otimes}}
\newcommand{\Id}{\operatorname{Id}}
\newcommand{\e}{\varepsilon}
\newcommand{\bs}{\setminus}
\newcommand{\In}{{\mathbf n}}
\newcommand{\invlim}{\varprojlim}

\begin{document}

\title{The topology of spaces of probability measures, I:\newline
Functors $\P_\tau$ and $\HP$.}
\shorttitle{The topology of spaces of probability measures, I}
\author{Taras Banakh}
\address{Faculty of Mechanics and Mathematics, Ivan Franko National University of L'viv, 1 Universytets'ka str., Lviv, Ukraine, 290602}

% AMS Subject Classification (2000)
%\keywords{\em }
\UDC{515.12}

\msyear{1995}
\msvolume{5}
\msnumber{1-2}
\mspages{65--87}
\received{01.10.1994}
\revised{24.01.1995}
\pageno{65}

\ruabstract{Т.Банах}{Топология пространств вероятностных мер, I:
Функторы $\P_\tau$ и $\HP$.}%
{}

\enabstract{T.Banakh}{Topology of probability measure spaces, I: functors $\P_\tau$ and $\HP$}%
{For a Tychonoff space $X$, the constructions $\HP(X)$ and $\Pt(X)$ of the spaces
of probability Radon measures and probability $\tau$-smooth measures on $X$
are considered.
It is proved that the constructions $\HP$ and $\Pt$ determine  functors
in the  category of Tychonoff spaces, which extend the functor $P$ of
probability measures in the category of compacta.
In this part we investigate
general topological properties of the spaces
$\HP(X)$ and  $\Pt(X)$, as well as categorial properties of the functors
$\HP$ and $\Pt$.}

\maketitle

\section*{Introduction}

The space of probability measures is a classical object which is studied form different points of view in Measure Theory, Functional Analysis, Probability Theory, Topology and Category Theory. This paper is the first part of a larger project (the results of which were announced in \cite{1}) devoted to the study of spaces of probability measures on topological spaces, in particular,  spaces of probability $\tau$-smooth measures and probability Radon measures. Our interests primarily touch  on topological and categorial aspects of Measure Theory and are very much in line with the survey \cite{2}, where a functor   $P:\Comp\to\Comp$ of the space of probability measures  in the category of compacts is studied (we are going to use contemporary terminology, understanding a compact Hausdorff space under the term "compact").

The study of spaces of probability measures leads to the problem of extension of the functor $P$ from the category of compacta to wider categories, in particular, the category $\Tych$ of Tychonoff spaces and their continuous maps. One of such extensions
$P_\beta$ was suggested by By A.Ch. Chigogidze \cite{3}: For a Tychonoff space $X$ let us consider the space $P_\beta(X)=\{\mu\in P(\beta X)\mid \supp (\mu)\subset
X\subset\beta\, X\}$, where $\beta\,X$ is the Stone-\v Cech compactification of
$X$,  and $\supp(\mu)$ is the support of the measure $\mu$. The structure $P_\beta(X)$
induces a functor $P_\beta:\Tych\to \Tych$, which extends the functor
$P:\Comp\to\Comp$. Another construction was considered in \cite{2} by V.V.~Fedorchuk,
who noted that the functor $P\circ \beta:\Tych\to\Comp$ assigning to each Tychonoff space $X$ the space $P(\beta X)$, also extends the functor $P:\Comp\to\Comp$.

However, the functors $P_\beta $ and $P\circ \beta $ have a number of drawbacks.
In particular, the space $P_\beta (X)$ is very narrow and does not contain many natural countably additive measures on $X$ (i.e. measures non-compact supports), and, on the other hand, the space $P(\beta X)$ is very broad,
and contains all finitely additive measures on $X$, and, as a result, the functor
$P\circ\beta$ does not preserve many specific properties of the space $X$,
in particular, it significantly raises the weight (although it does not raise the density).

Thus, it is natural to consider spaces of measures, which lie between the spaces $P_\beta(X)$ and $P(\beta X)$.

For that purpose, let us consider for a Tychonoff space $X$ the following two spaces of probability measures:
$$
\hat P(X)=\{\mu\in P(\beta X)\mid\mu_*(X)=1\}\quad\text{ and }\quad
P_\tau(X)=\{\mu\in P(\beta X)\mid \mu^*(X)=1\},
$$
where $\mu_*(X)=\sup\{\mu(B)\mid X\supset B$ is a Borel subset of
$\beta X\}$ and $\mu^*(X)=\inf\{\mu(B)\mid X\subset B$ is a  Borel subset of
$\beta X\}$, which are, correspondingly, the upper and the lower  $\mu$-measures of the set $X$ in
$\beta X$ (as a tribute to historical tradition, we use the notation $\hat
P(X)$ and $P_\tau(X)$, not $P_*(X)$ and $P^*(X)$, which seem to be more natural). Evidently, $$P_\beta(X)\subset \hat P(X)\subset
\Pt(X)\subset P(\beta X)$$ for any Tychonoff space $X$, and
$P_\beta(X)=\hat P(X)=\Pt(X)=P(\beta X)$, if the space $X$ is compact.

The measures belonging to the spaces $\HP(X)$ and $\Pt(X)$ can be equivalently described both in terms of countably additive measures on the space $X$ and in terms of linear functionals on the Banach space $C_b(X)$ of bounded continuous real-valued functions on
$X$. Before providing exact formulations, let us recall some definitions.

A countably additive finite measure $\mu$, defined on the $\sigma $-algebra
${\mathcal B}(X)$ of Borel subsets of a topological space $X$,
is called
\begin{itemize}
\item[(i)] {\it a probability measure}, if $\mu(X)=1$;
\item[(ii)] {\it a regular measure}, if $\mu(A)=\sup\{\mu(Z)\mid A\supset Z$ is a closed subset of
 $X\}$ for every Borel subset
$A\subset X$;
\item[(iii)] {\it Radon}, if $\mu(A)=\sup\{\mu(K)\mid A\supset K$ is a compact subset of $X\}$ for every Borel subset
$A\subset X$ (in \cite{4} Radon measures are called dense measures);
\item[(iv)] {\it $\tau$-smooth}, if for any monotonically decreasing net  $\{Z_\alpha \}$ of closed subsets of $X$ with empty intersection $\bigcap_\alpha  Z_\alpha $, the net
$\{\mu(Z_a)\}$ of real numbers converges to zero (see \cite{4}).
\end{itemize}

Further in  this text, by measure on a topological space we shall understand a finitely additive Borel measure. One easily notices that every Radon measure on a Hausdorff space is regular and $\tau$-smooth. Moreover, a regular measure $\mu$ on a Hausdorff space $X$ is a Radon measure if and only if $$\mu(X)=\sup\{\mu(K)\mid \mbox{$K$ is a compact subset of $X$}\}.$$

Let $X$ be a Tychonoff space. For every measure $\mu\in \Pt(X)$
we will define a measure $\tilde \mu$  on $X$ by the formula $\tilde
\mu(A)=\mu^*(A)=\inf\{\mu(B)\mid A\subset B$ is a Borel subset of
$\beta X\}$, where $A$ is a Borel subset of $X$. It is known \cite{5},
or \cite[1.11]{2} (see~also Remark~\ref{r1.2}) that the measure $\tilde \mu$, defined in this way, is  $\tau$-smooth on $X$. Conversely, every probability $\tau$-smooth measure
$\tilde \mu$ on $X$ determines a measure $\mu\in \Pt(X)$ by the formula
$\mu(A)=\tilde \mu(A\cap X)$, where $A\in{\mathcal B}(\beta X)$. Under this condition Radon measures (and only them) become measures on $\beta X$,
which belong to the set $\HP(X)$. Therefore, we will call measures from $\HP(X)$ {\em Radon measures}, and the measures from $\Pt(X)$  {\em $\tau$-smooth measures}.

By $C_b(X)$ we denote the Banach space of all bounded continuous real-valued functions on $X$, endowed with the norm
$\|f\|=\sup\{|f(x)|:x\in X\}$, $f\in C_b(X)$. Not getting into the definition of the integral, let us note that every regular probability measure $\mu$ on $X$
uniquely determines the integral $\int_\mu$, a non-negative linear functional  on $C_b(X)$ of norm 1 (the value of integral $\int_\mu$
on the function $f$ will be denoted by $\int_\mu f$, or simply $\mu(f)$ ). Under such identification  we get the following equivalences (see \cite{4}),
\begin{itemize}
\item[(i)] the measure $\mu$ is $\tau$-dense if and only if
$\mu(f_\alpha )\to 0$ for any monotonically decreasing net
$\{f_\alpha \}\subset C_b(X)$ that pointwise  converges to zero;
\item[(ii)] the measure $\mu$ is Radon if and only if
$\mu(f_\alpha )\to 0$ for any net of uniformly bounded sequence of functions $\{f_\alpha \}\subset
C_b(X)$ that converges to zero uniformly on compacta.
\end{itemize}

Similarly, a measure $\mu\in P(\beta X)$ belongs to the set $\Pt(X)$ if and only if
 $\mu(f_\alpha )\to 0$ for every monotonically decreasing net
$\{f_\alpha \}\subset C(\beta X)$ converging to zero pointwise on the set
$X\subset \beta X$.

Let us show that the constructions of the spaces $\Pt(X)$ and $\HP(X)$ are functorial in the category $\Tych$. Since $\HP(X)\subset \Pt(X)\subset P(\beta X)$ for
any Tychonoff space $X$ and $P\circ \beta :\Tych\to\Comp$ is a functor on the category $\Tych$ of Tychonoff spaces \cite{2}, in order to see that the constructions  $\HP$ and $\Pt$ are functorial, it is sufficient to show that for any continuous map $f:X\to Y$ of Tychonoff spaces
$P(\beta f)(\Pt(X))\subset \Pt(Y)$ and $P(\beta f)(\HP(X))\subset \HP(Y)$,
where $\beta f:\beta X\to\beta Y$ is the Stone-\v Cech compactification of the map $f$ (see \cite[3.66]{6}). If $\mu\in\Pt(X)$, then $\mu^*(X)=1$, and,
consequently, $\mu(B)=1$ for any Borel set $B$,
$X\subset B\subset \beta X$. Then for any Borel set $B'$,
$Y\subset B'\subset \beta Y$, $P(\beta f)(\mu)(B')=\mu((\beta
f)^{-1}(B'))=1$, since $(\beta f)^{-1}(B')$ is a Borel subset of
$\beta X$, which contains $X$. This implies that $P(\beta f)(\mu)\in\Pt(Y)$,
i.e. $P(\beta f)(\Pt(X))\subset\Pt(Y)$.

If $\mu\in\HP(X)$, then for any $\e>0$ there exists a compact $K\subset X$, such that $\mu(K)>1-\e$.
Then $f(K)\subset Y$ is a compact in $Y$ such that $P(\beta
f)(\mu)(f(K))=\mu((\beta f)^{-1}(f(K)))\ge \mu(K)>1-\e$. Therefore,
the measure $P(\beta f)(\mu)$  belongs to the set  $\HP(Y)$, i.e. $P(\beta
f)(\HP(X))\subset \HP(Y)$. Let us define $\Pt(f)=P(\beta
f)|\Pt(X):\Pt(X)\to\Pt(Y)$ and $\HP(f)=P(\beta f)|\HP(X):\HP(X)\to \HP(Y)$.
Thus, we have proved the following

\begin{theorem}\label{t0.1.} The constructions $\Pt$ и $\HP$ are covariant functors in the category $\Tych$ of Tychonoff spaces and their continuous maps, which extend the functor $P:\Comp\to\Comp$.
\end{theorem}

Let us note that we might as well have defined the functors $\Pt:\Tych\to\Tych$ and
$\HP:\Tych\to\Tych$ from inside, without using Stone-\v Cech compactifications. In particular, for the Tychonoff space $X$ the space
$\Pt(X)$ consists of regular probability  $\tau$-smooth measures on $X$, and the topology on
 $\Pt(X)$ is induced by a subbase consisting of sets of the form
$\{\mu\in\Pt(X) : |\mu(\varphi)-\mu_0(\varphi)|<1\}$, where $\mu_0\in\Pt(X)$
and $\varphi\in C_b(X)$. If $f:X\to Y$ is a continuous map of Tychonoff spaces, then the map $\Pt(f):\Pt(X)\to\Pt(Y)$ is defined by the formula $\Pt(f)(\mu)(A)=\mu(f^{-1}(A))$, where $\mu\in\Pt(X)$ and $A$ is a Borel subset of $Y$. Then $\HP(X)$ is a subspace of
$\Pt(X)$, consisting of Radon probability measures, and $\HP(f)$ is the restriction of the map $\Pt(f)$ to the set $\HP(X)$. Using the roundabout way (using Stone-\v Cech compactifications) we got rid of the necessity of checking that the constructions $\Pt$ and $\HP$, defined is such way, are functors in the category $\Tych$ indeed.

Before moving on to the presentation of concrete results, let us note that in a number of spaces, for example, spaces $X$ which are Borel sets in their Stone-\v Cech compactification, every $\tau$-smooth measure is Radon. In this case the spaces $\Pt(X)$ and $\HP(X)$ coincide. More generally, this holds for the so called universally measurable spaces, that is, spaces $X$ which are measurable in some compactification $\gamma X$ with respect to any measure  $\mu\in P(\gamma X)$. Besides absolute Borel spaces, analytic and coanalytic spaces are also absolutely measurable \cite[2.2.12]{7}.

\section{Categorial properties of the functor $\Pt$}

In this section we shall investigate  categorial properties of the functor $\Pt$, and also some general topological properties of the spaces $\Pt(X)$.

Let us start with the following simple remark.

\begin{lemma}\label{l1.1} Let $X$ be a Tychonoff space. If
$\mu\in\Pt(X)$, then $\mu(A)=\mu(B)$ for any two Borel subsets $A,B\subset \beta X$ such that $A\cap X=B\cap X$.
\end{lemma}

\begin{proof} Let $A,B\subset \beta X$ be Borel sets with $A\cap X=B\cap X$. Then
$$
\begin{aligned}
|\mu(A)-\mu(B)|&=|\mu(A\cap
B)+\mu(A\bs B)-\mu(A\cap B)-\mu(B\bs A)|=\\
&=|\mu(A\bs B)-\mu(B\bs A)|\le
\mu(A\bs B)+\mu(B\bs A)=\mu((A\bs B)\cup(B\bs A)).
\end{aligned}
$$ Since $A\cap
X=B\cap X$, we have that $A\triangle B=(A\bs B)\cup (B\bs A)\subset \beta X\bs X$.
If $\mu\in \Pt(X)$, then $\mu_*(\beta X\bs X)=0$, which implies that
$\mu(A\triangle B)=\mu_*(A\triangle B)\le \mu_*(\beta X\bs X)=0$, which means that $|\mu(A)-\mu(B)|\le \mu(A\triangle B)=0$, i.e. $\mu(A)=\mu(B)$.
Thus, the lemma is proved.
\end{proof}

\begin{remark}\label{r1.2} Lemma \ref{l1.1} implies the following fact, which has already been mentioned in the introduction: every measure $\mu\in\Pt(X)$ induces a probability measure $\tilde \mu$ on $X$, according to the formula $\tilde
\mu(A)=\mu(B)$, where $B$ is any Borel subset $\beta X$ such that $B\cap X=A$, and $A$ is a Borel subset of $X$. Under these conditions, the measure
$\tilde \mu$ is $\tau$-smooth.
Indeed, for every monotonically decreasing net
$\{Z_\alpha\}$ of non-empty closed subsets of $X$ with empty intersection, the net $\{\bar Z_\alpha \}$ consisting of their closures in $\beta X$ also monotonically decreases. Then, as $\{\bar Z_\alpha \}$ is a centered family of closed subsets of the compact $\beta X$, it has a non-empty intersection $Z=\bigcap_\alpha \bar Z_\alpha $ \cite[3.1.1]{6}. Since $Z\cap
X=(\cap_\alpha \bar Z_\alpha )\cap X=\cap_\alpha (\bar Z_\alpha \cap
X)=\cap_\alpha Z_\alpha =\emptyset$, we have that $Z\subset \beta X\bs X$.
Keeping in mind that $\mu\in\Pt(X)$, we get that $\mu(Z)=0$. The regularity of the measure
$\mu$ implies that for any $\e>0$ there exists an open set
$U$, $Z\subset U\subset \beta X$ such that $\mu(U)<\e$. \cite[3.1.5]{6} implies that since $\cap_\alpha \bar Z_\alpha =Z\subset U$, $Z_{\alpha
_0}\subset U$ for some $\alpha _0$. Consequently, $\mu(\bar
Z_{\alpha _0})\le \mu(U)<\e$, and, therefore, $\tilde\mu(Z_{\alpha _0})=\mu(\bar
Z_{\alpha _0})<\e$. Since the net $\{Z_\alpha \}$ monotonically decreases, $\tilde \mu(Z_\beta )\le \tilde \mu(Z_{\alpha _0})<\e$ for all $\beta \ge\alpha _0$. But this means that the real-valued net $\{\tilde
\mu(Z_\alpha )\}$ converges to zero, thus, the measure $\tilde \mu$ on $X$ is
$\tau$-smooth.
\end{remark}

Let us recall that a map $f:X\to Y$ between topological spaces is called {\em perfect}, if it is closed and the preimage $f^{-1}(y)$ of every point $y\in
Y$ is compact.

\begin{theorem}\label{t1.3} The functor $\Pt:\Tych\to\Tych$ preserves the class of perfect maps.
\end{theorem}

\begin{proof} Let $f:X\to Y$ be a perfect map of Tychonoff spaces. Then the extension $\beta  f:\beta X\to \beta Y$ of the map $f$ (called the {\em Stone-\v Cech compactification} of the map $f$) has the following property:
$\beta f(\beta X\bs X)\subset \beta Y\bs Y$ \cite[3.7.15]{6}.
Let us consider the map $P(\beta f):P(\beta X)\to P(\beta Y)$. We are going to prove that
$P(\beta f)(P(\beta X)\bs \Pt(X))\subset P(\beta Y)\bs \Pt(Y)$.
Indeed, let $\mu\in P(\beta X)\bs \Pt(X)$, i.e. $\mu^*(X)<1$. This implies that there exists a compact $K\subset \beta X\bs X$ such that
$\mu(K)>0$. Then $\beta f(K)\subset \beta Y\bs Y$ is  a compact subset with
$P(\beta f)(\mu)(\beta f(K))=\mu((\beta f)^{-1}(\beta f(K)))\ge \mu(K)>0$.
Consequently, $P(\beta f)(\mu)^*(Y)<1$, i.e. $P(\beta f)(\mu)\notin
\Pt(Y)$. Thus, $P(\beta f)(P(\beta X)\bs \Pt(X))\subset P(\beta
Y)\bs \Pt(Y)$.
Since $P(\beta f):P(\beta X)\to P(\beta Y)$ is a map between compacta,
the lasts inclusion implies that the map
$\Pt(f)=P(\beta f)|\Pt(X):\Pt(X)\to\Pt(Y)$ is perfect.
The theorem is proved.
\end{proof}

\begin{theorem}\label{t1.4} The functor $\Pt:\Tych\to\Tych$ preserves the class of embeddings.
\end{theorem}

\begin{proof} Let $f:X\to Y$ be a topological embedding of Tychonoff spaces and  $\beta  f:\beta X\to \beta
Y$ be its Stone-\v Cech compactification. One can easily see that
$\beta f(\beta X\bs X)\subset \beta Y\bs
f(X)$. Let $A=\{\mu\in P(\beta Y)\mid \mu^*(f(X))=1\}$.
Similarly to the proof of Theorem 1.3, it can be shown that
$P(\beta f)(P(\beta X)\bs \Pt(X))\subset P(\beta Y)\bs A$. Obviously,
$P(\beta f)(\Pt(X))\subset A$. Thus, the map $\Pt(f)=P(\beta
f)|\Pt(X):\Pt(X)\to A$ is proper. Let us show that it is also injective, which will imply that  $\Pt(f):\Pt(X)\to \Pt(Y)$ is an embedding.

Let $\mu,\eta\in\Pt(X)$ be two distinct measures. The there exists a closed set $Z\subset \beta X$ such that $\mu(Z)\ne \eta(Z)$. We state that $\Pt(f)(\mu)(\beta f(Z))\ne \Pt(f)(\eta)(\beta f(Z))$,
which will imply that the measures $\Pt(f)(\mu), \Pt(f)(\eta)\in\Pt(Y)$
are distinct. Indeed, letting $Z'=(\beta f)^{-1}(\beta f(Z))$,
let us observe that, by the definition, $\Pt(f)(\mu)(\beta f(Z))=\mu(Z')$ and
$\Pt(f)(\eta)(\beta f(Z))=\eta(Z')$. Since $f$ is an embedding, $Z'\cap
X=Z\cap X$. Then by Lemma 1.1 $$\Pt(f)(\mu)(\beta
f(Z))=\mu(Z')=\mu(Z)\ne \eta(Z)=\eta(Z')=\Pt(f)(\eta)(\beta f(Z)),$$ i.e.
the measures $\Pt(f)(\mu),\Pt(f)(\eta)\in\Pt(Y)$ are distinct.  The theorem is proved.
\end{proof}

Theorems \ref{t1.3}, \ref{t1.4} immediately imply

\begin{corollary}\label{c1.5} The functor $\Pt:\Tych\to\Tych$ preserves the class of closed embeddings.
\end{corollary}

Since the functor $\Pt$ preserves embeddings, for a pair $X\subset Y$ of Tychonoff spaces we will treat the space $\Pt(X)$ as a subset $\{\mu\in \Pt(Y)\mid \mu^*(X)=1\}$ of $\Pt(Y)$. Let us note that
when we do this, the set $\HP(X)\subset \Pt(X)$ consisting of Radon probability measures on $X$ becomes the subset
$\{\mu\in\Pt(Y)\mid \mu_*(X)=1\}\subset \Pt(Y)$. It is also worthy of note that Theorem 1.4 implies that the construction of the space $\Pt(X)$ in fact does not depend on the compactification of $X$, i.e. for any compactification $\gamma X$ of the $X$ the space $\{\mu\in P(\gamma
X)\mid
\mu^*(X)=1\}$ is naturally homeomorphic to $\Pt(X)$. As we will see in \S 2, the image of the set under this homeomorphism $\{\mu\in P(\gamma X)\mid \mu_*(X)=1\}$ is the space $\HP(X)$ of Radon probability measures on $X$.
\vskip3pt

Let us recall that the support of a measure $\mu\in P(X)$ on a compact space $X$
is the set $\supp(\mu)=\cap\{F\mid F$, a closed subset of
$X$ such that $\mu(F)=1\}$. Under this condition $\mu(\supp(\mu))=1$, i.e. the support of the measure $\mu$ is the smallest closed set of $\mu$-measure one. If $X$ is a Tychonoff space, then by the support of a  $\tau$-smooth probability measure $\mu\in\Pt(X)$ on $X$ we will sometimes understand the set $\supp(\mu)\cap X$.
\vskip3pt

The functor $\Pt$ preserves neither injective nor surjective maps. To see that $\Pt$ does not preserve injective maps, choose any non-measurable subset $Z\subset Y$ in the closed interval $Y=[0,1]$ with lower and upper Lebesgue measures $\lambda_*(Z)=0$ and $\lambda^*(Z)=1$. Next, consider the subspace $X=Z\times\{0\}\cup ([0,1]\setminus Z)\times\{1\}$ of the plane $\IR^2$, and let $f:X\to Y$,  $f:(z,t)\mapsto z$,  be the projection onto the first coordinate. It is clear that the map $f$ is bijective and continuous. It can be shown (see \cite[Example 3]{BSF}) that the Lebesgue measure $\lambda\in P(Y)=\Pt(Y)$ has two preimages under the map $\Pt(f):\Pt(X)\to\Pt(Y)$, which means that the map $\Pt(f)$ is not injective and hence $\Pt$ does not preserve injective maps\footnote{This example was added at the translation.}.

To see that the functor $\Pt$ does not preserve surjective maps, consider the bijective map $f:D\to [0,1]$ of a discrete space $D$ onto the interval $[0,1]$.
Then the standard Lebesque measure on $[0,1]$ does not have a preimage under the map $\Pt(f):\Pt(D)\to P[0,1]$. This is due to the fact that the set
$D$, being open in its Stone-\v Cech compactification, is measurable with respect to any measure $\mu\in P(\beta D)$. Consequently, any
$\tau$-smooth measure on $D$ is Radon and, since the space
$D$ is discrete, it is also atomic (on atomic measures see \cite[\S 2]{8}). But the image of an atomic measure under the map $\Pt(f)$ is also an atomic measure, and, consequently, it is not equal to the Lebesque measure on $[0,1]$.
\vskip3pt

Nonetheless, the functor $\Pt$ preserves one of the properties of maps, which in the case of compactness implies its surjectivity.

\begin{proposition}\label{p1.6} Let $f:X\to Y$ be a map with dense image $f(X)$ in $Y$. Then the image $\Pt(f)(\Pt(X))$ is dense in $\Pt(Y)$.
\end{proposition}

\begin{proof} One can easily see that $\Pt(f)(\Pt(X))$ contains the set  $$P_\omega(f(X))=\{\mu\in Y: |\supp(\mu)|<\infty$$ and
$\supp(\mu)\subset f(X)\}$, which is dense in $P(\beta Y)\supset\Pt(Y)$.
\end{proof}

\begin{theorem}\label{t1.7} The functor $\Pt$ preserves preimages, i.e. for any map $f:X\to Y$ between Tychonoff spaces and any subset
$A\subset Y$ we have that $\Pt(f)^{-1}(\Pt(A))=\Pt(f^{-1}(A))$.
\end{theorem}

\begin{proof} The inclusion $\Pt(f^{-1}(A))\subset \Pt(f)^{-1}(\Pt(A))$
is trivial. Let us show that $\Pt(f)^{-1}(\Pt(A))\subset \Pt(f^{-1}(A))$.
This can be derived from the inclusion $$\Pt(f)(\Pt(X)\bs \Pt(f^{-1}(A)))\subset \Pt(Y)\bs \Pt(A).$$ Let
$\mu\in \Pt(X)\bs \Pt(f^{-1}(A))$, i.e. $\mu^*(f^{-1}(A))<1$. This means that there exists a compact $K\subset X\bs f^{-1}(A))$, such that $\mu(K)>0$.
Then $f(K)$ is a compact subset of $Y\bs A$ such that $\Pt(f)(\mu)(f(K))=
\mu(f^{-1}(f(K)))\ge \mu(K)>0$, i.e. $\Pt(f)(\mu)(A)<1$, and, as a consequence,
$\Pt(f)(\mu)\notin \Pt(A)$. The theorem is proved.
\end{proof}

An embedding-preserving functor $F:\Tych\to\Tych$ is said to preserve ({\em closed}) {\em intersections}, if for any Tychonoff space $X$ and a family $\{X_\alpha \}_{\alpha \in A}$ of its (closed) subsets we get $F(\bigcap_{\alpha \in A}X_\alpha )=\bigcap_{\alpha \in
A}F\,X_\alpha $.

\begin{remark}\label{r1.8} Unlike the functor $\HP$, which preserves countable intersections (see Theorem \ref{t2.15}), the functor $\Pt$ does not preserve even finite intersections. This can be seen from the following example: let $X\subset
[0,1]$ be a subset of the interval such that $\lambda ^*(X)=1$ and $\lambda _*(X)=0$,
where $\lambda $ is a the standard Lebesque measure on $[0,1]$. Then $\lambda
^*([0,1]\bs X)=1$. Consequently, $\lambda \in \Pt(X)\cap \Pt([0,1]\bs
X)$. However, $\Pt(X\cap ([0,1]\bs X))=\Pt(\emptyset)=\emptyset$.
\end{remark}

Yet, we have the following

\begin{proposition}\label{p1.9} Let $X$ be a Tychonoff space and
$A,B\subset X$ -- its two subsets, one of which is Borel. Then
$\Pt(A\cap B)=\Pt(A)\cap\Pt(B)$.
\end{proposition}

\begin{proof}  Without loss of generality we can assume that
$B\subset X$ is Borel. Let $\tilde B\subset \beta X$ be a Borel subset of $\beta X$ such that $\tilde B\cap X=B$. Apparently, $\Pt(A\cap B)\subset \Pt(A)\cap \Pt(B)\subset \Pt(A)\cap \Pt(\tilde B)$.
We will show that the inverse inclusion also holds. Let us fix a measure
$\mu\in\Pt(A)\cap\Pt(\tilde B)$ and note that $A\cap B=A\cap\tilde B$.
To show that $\mu\in \Pt(A\cap B)$ it is sufficient to show that
$\mu^*(A\cap\tilde B)=1$. Let $K\subset \beta X$ be an arbitrary compact subset such that $K\subset \beta X\bs(A\cap\tilde B)=(\beta X\bs A)\cup(\beta X\bs\tilde B)$.
Our aim is to prove that $\mu(K)=0$. Let us present the compact $K$ as a union $K=K_1\cup K_2$ of two Borel sets $K_1=K\bs\tilde B$ and
$K_2=K\cap \tilde B$. Since $\mu\in\Pt(\tilde B)$, $\mu(K_1)=0$.
Then, let us note that $K_2\subset \beta X\bs A$. As $\mu\in \Pt(A)$ and
$K_2$ is a Borel set such that $K_2\cap A=\emptyset$,
$\mu(K_2)=0$. Therefore, $\mu(K)=\mu(K_1)+\mu(K_2)=0$. Thus,
$\mu\in \Pt(A\cap B)$, which means that $\Pt(A\cap B)=\Pt(A)\cap \Pt(B)$.
The claim is proved.
\end{proof}

\begin{theorem}\label{t1.10} The functor $\Pt:\Tych\to\Tych$ preserves intersections of closed subsets, i.e. for any Tychonoff space $X$ and closed subsets $X_\alpha\subset X$, $\alpha\in A$, we get $$\Pt(\bigcap_{\alpha\in A}X_\alpha)=\bigcap_{\alpha\in
A}\Pt(X_\alpha).$$
\end{theorem}

\begin{proof} The inclusion $\Pt(\bigcap_{\alpha\in A}X_\alpha)\subset
\bigcap_{\alpha\in A}\Pt(X_\alpha)$ is obvious. Let $\mu\in\bigcap_{\alpha\in
A}\Pt(X_\alpha)$. In order to prove that $\mu\in\Pt(\cap_{\alpha \in A}X_\alpha
)$, we need to show that  $\mu^*(\cap_{\alpha \in A}X_\alpha )=1$, or, equivalently, that for every Borel set $B\subset \beta X$,
$\bigcap_{\alpha \in A}X_\alpha \subset B$, we have that $\mu(B)=1$. Let
$\bar X_\alpha $ be the closure of the set $X_\alpha $ in $\beta X$. As the set
 $X_\alpha$ is closed in $X$, $\bar X_\alpha \cap X=X_\alpha $.
Since $\mu\in\Pt(X_\alpha )$ for every  $\alpha $, $\mu(\bar
X_\alpha )=\mu^*(\bar X_\alpha )=1$. Therefore, $\supp(\mu)\subset
\bigcap_{\alpha \in A}\bar X_\alpha $. Since $\mu(\supp(\mu))=1$, we have that
$\mu(\cap_{\alpha \in A}\bar X_\alpha )=1$. Since the set
$\bigcap_{\alpha \in A}\bar X_\alpha \subset \beta X$ is closed and
$\bigl(\cap_{\alpha \in A}\bar X_\alpha \bigl)\cap X=\bigcap_{\alpha \in
A}(\bar X_\alpha \cap X)=\bigcap_{\alpha \in A}X_\alpha $, for every Borel set $B$, $\bigcap_{\alpha \in A}X_\alpha \subset
B\subset \beta X$, according to Lemma 1.1, $\mu(B)\ge\mu(\cap_{\alpha \in
A}\bar X_\alpha )=1$. Hence $\mu^*(\cap_{\alpha \in A}X_\alpha )=1$,
i.e. $\mu\in\Pt(\cap_{\alpha \in A}X_\alpha )$. The theorem is proved.
\end{proof}

Now let us consider the question of continuity of the functor $\Pt$. Let $A$ be a directed partially ordered set (a partially ordered set is {\em directed} if for any $\alpha,\beta\in A$ there exists a $\gamma\in A$, such that
$\gamma \ge \alpha$ и $\gamma\ge\beta$ ).

Let $\{X_\alpha,p_\alpha^\beta\}$ be an inverse system indexed be the set
$A$ and consisting of Tychonoff spaces. By $\invlim X_\alpha$ we will denote the limit of this system, and by $p_\alpha:\invlim X_\alpha\to X_\alpha$,
$\alpha\in A$, its bonding mappings (limit projections).
The inverse system $\{X_\alpha,p_\alpha^\beta\}$ induces an inverse limit
$\{\Pt(X_\alpha),\Pt(p_\alpha^\beta)\}$, whose limit will be denoted by
$\invlim \Pt(X_\alpha)$, and limit projections by $\pr_\alpha:\invlim
\Pt(X_\alpha)\to\Pt(X_\alpha)$. The map $\Pt(p_\alpha):\Pt(\invlim
X_\alpha)\to \Pt(X_\alpha)$ induces a map  $R:\Pt(\invlim X_\alpha)\to \invlim \Pt(X_\alpha)$.

It is well-known that if all $X_\alpha$ are compact, then the map $R$ is a homeomorphism. This implies from the continuity of the functor $P$ in the category of compacta \cite[VII.3.11]{9}.

\begin{theorem}\label{t1.11} The map $R:\Pt(\invlim X_\alpha)\to \invlim
\Pt(X_\alpha)$ is an embedding. If the bonding mappings (limit projections)
$p_\alpha:\invlim X_\alpha \to X_\alpha$ are dense (i.e. $p_\alpha(\invlim X_\alpha)$
is dense in $X_\alpha$), then the image $R(\Pt(\invlim
X_\alpha))$ is dense in $\invlim \Pt(X_\alpha)$.
\end{theorem}

\begin{proof} Let us consider the Stone-\v Cech compactification $\{\beta
X_\alpha,\beta(p_\alpha^\beta)\}$ of the system $\{X_\alpha,p_\alpha^\beta\}$ and note that
$\invlim X_\alpha$ can be embedded in $\invlim \beta X_\alpha$.
Moreover, if the bonding mappings $p_\alpha:\invlim X_\alpha\to X_\alpha$ are  dense, then the image of the space $\invlim X_\alpha$ is dense in $\invlim\beta X_\alpha $.
By the continuity of the functor $P:\Comp\to\Comp$, the corresponding map $\bar R:P(\invlim
\beta X_\alpha)\to \invlim P(\beta X_\alpha)$ is a homeomorphism. Applying the fact that the functor $\Pt$ preserves embeddings, we get that the map
$R:\Pt(\invlim X_\alpha)\to \invlim \Pt(X_\alpha)$ can be embedded in the homeomorphism
$\bar R$ and is, therefore, an embedding.
Furthermore, since the functor $\Pt$ preserves maps with dense images, given that the bonding mappings $p_\alpha$ are dense, the image of the space
$\Pt(\invlim X_\alpha)$ under the embedding $R$ is dense in $\invlim\Pt(X_\alpha)$.
The theorem is proved.
\end{proof}

Now we are going to consider the property of preserving homotopies, which, in the compact case, is closely linked to the continuity of functors \cite{10}.

For Tychonoff spaces $X$ and $Y$ let us define a map $j_{XY}:\Pt(X)\times Y\to
\Pt(X\times Y)$ determined by the formula
$j_{XY}(\mu,y)=\Pt(i_y)(\mu)$, $\mu\in \Pt(X)$, $y\in Y$, where $i_y:X\to
X\times Y$ is an embedding of $X$ in $X\times Y$ as a fiber:
$i_y(x)=(x,y)$, $x\in X$.

\begin{proposition}\label{p1.12} The map $j_{XY}:\Pt(X)\times Y\to \Pt(X\times
Y)$ is a closed embedding.
\end{proposition}

\begin{proof} Let $X, Y$ be Tychonoff spaces and $\beta X$
и $\beta Y$ be their Stone-\v Cech compactifications. According to \cite[VII.5.11 and
VII.5.18]{9}, the map $j_{\beta X, \beta Y}:P(\beta X)\times \beta Y\to
P(\beta\times \beta Y)$ is an embedding of compacta. Now the claim follows from the obvious equality $$j_{\beta X,\beta Y}(\Pt(X)\times
Y)=j_{\beta X,\beta Y}(P(\beta X)\times \beta Y)\cap \Pt(X\times Y).$$
\end{proof}

\begin{corollary}\label{c1.13} The functor $\Pt$ preserves homotopies, i.e. for any homotopy $H_t:X\to Y$ the homotopy $\Pt(H_t):\Pt(X)\to \Pt(Y)$ is continuous as a map $\Pt(H_{(\cdot)}):\Pt(X)\times [0,1]\to \Pt(Y)$.
\end{corollary}

\begin{proof} Let $H:X\times[0,1]\to Y$ be a homotopy. Then the map $\Pt(H_{(\cdot)}):\Pt(X)\times [0,1]\to \Pt(Y)$ is continuous as a composition $\Pt(H_{(\cdot)})=\Pt(H)\circ j_{X,[0,1]}$ of continuous maps $j_{X,[0,1]}:\Pt(X)\times [0,1]\to \Pt(X\times [0,1])$ and
$\Pt(H):\Pt(X\times [0,1])\to \Pt(Y)$.
\end{proof}

Let us recall the definition of a natural transformation of functors. Let
$F_i:\mathcal C\to \mathcal C'$, $i=1,2$ be two covariant functors from a category
${\mathcal C}=(\mathcal O,\mathcal M)$ to a category $\mathcal C'=(\mathcal O',\mathcal M')$. A family $\Phi=\{\varphi_X:F_1(X)\to F_2(X),\; X\in\mathcal O\}\subset\mathcal M'$ of morphisms is called a natural transformation of the functor $F_1$ to the functor $F_2$, if for any morphism $f:X\to Y$ in the category $\mathcal C$ the following diagram is commutative:
$$
\begin{CD}
F_1(X)@>{\varphi_X}>> F_2(X)\\
@V{F_1(f)}VV @ VV {F_2(f)}V\\
F_1(Y)@>{\varphi_Y}>> F_2(Y).
\end{CD}
$$

For every Tychonoff space $X$ let us define a map $\delta_X:X\to \Pt(X)$ which assigns to each point $x\in X$  the Dirac measure
$\delta_X(x)$ concentrated at the point $x$.

\begin{theorem}\label{t1.14} The family $\delta=\{\delta_X\}$ defines a natural transformation of the identity functor $\Id:\Tych\to \Tych$ to the functor $\Pt:\Tych\to\Tych$, and, moreover, every component
$\delta_X:X\to\Pt(X)$ is a closed embedding.
\end{theorem}

\begin{proof} One can easily check that $\delta=\{\delta_X\}$ is a natural transformation of the functor $\Id$ to the functor $\Pt$. The fact that every map $\delta_X:X\to \Pt(X)$ is a closed embedding implies from
\cite[II, \S3]{4}.
\end{proof}

\begin{theorem}\label{t1.15} The functor $\Pt$ does not increase the density of Tychonoff spaces, i.e. $d(\Pt(X))\le d(X)$ for any infinite Tychonoff space $X$.\footnote{${}^2$ However, there exists a non-separable compact Hausdorff space $X$ with separable space $P(X)$ of probability measures, see  \cite{Tal80}, \cite{DP08}, \cite{APR14}.}
\end{theorem}

\begin{proof} Let $A\subset X$ be a dense subset of cardinality
$d(X)$ of an infinite Tychonoff space $X$. Then the set $B=\{\sum_{i=1}^nr_i\delta(x_i)\mid n\in\IN$ and for every $1\le i\le
n$ \ $r_i$ is rational and
$x_i\in A\}$ is dense in $\Pt(X)$. Moreover, it is clear that  the cardinality of the set $B$ is equal to $d(X)$.
\end{proof}

\begin{theorem}\label{t1.16} The functor $\Pt$ preserves the weight of Tychonoff spaces, i.e. $w(\Pt(X))=w(X)$ for any infinite Tychonoff space
$X$.
\end{theorem}

\begin{proof} Let $X$ be an infinite Tychonoff space. According to \cite[3.5.2]{6}, there exists a compactification $c X$ of the space
$X$ such that $w(cX)=w(X)$. By Theorem 1.4, the space $\Pt(X)$ can be embedded in the compact space $\Pt(cX)=P(cX)$. Since the functor $P$ preserves  weight \cite[VII.3.9]{9}, $w(P(cX))=w(cX)=w(X)$. Consequently, $w(\Pt(X))\le w(X)$, and, since $X$ can be embedded in $\Pt(X)$, $w(X)\le w(\Pt(X))$. That is,
$w(\Pt(X))=w(X)$. The theorem is proved.
\end{proof}

The following theorem follows from \cite[II, \S4]{4} implies

\begin{theorem}\label{t1.17} The functor $\Pt$ preserve the class of metrizable spaces.
\end{theorem}

Let us recall that a $p$-{\em paracompact} space is a preimage of a metrizable space under a perfect map \cite{11}. Theorems \ref{t1.3} and \ref{t1.7} imply

\begin{theorem}\label{t1.18} The functor $\Pt$ preserves the class of $p$-paracompact space.
\end{theorem}

Let $X$ be a topological space. For every countable ordinal number
$\alpha$ we will define families $\F_\alpha(X)$ и $\G_\alpha(X)$ of Borel sets in the following way: the family $\F_0(X)$ (family $\G_0(X)$) consists of all closed (open) subsets of the space $X$, the family
$\F_\alpha(X)$ (the family $\G_\alpha(X)$) consists of all countable unions (countable intersections) of sets from $\bigcup_{\xi<\alpha}\F_\xi(X)$ (from
$\bigcup_{\xi<\alpha}\G_\xi(X)$~) for odd ordinals $\alpha$ and all countable intersections (countable unions) of the sets from
$\bigcup_{\xi<\alpha}\F_\xi(X)$  (from $\bigcup_{\xi<\alpha}\G_\xi(X)$~) for even ordinals. It is obvious that for any $A\in\F_\xi(X)$ ( $A\in \G_\xi(X)$~)
we have that $X\bs A\in\G_\xi(X)$   ($X\bs A\in\F_\xi(X)$~).

For a topological space $X$ by $\mathcal B_0(X)$ we denote the
$\sigma$-algebra of all Baire subsets of $X$, i.e. the smallest
$\sigma$-algebra containing all functionally closed subsets of $X$. Baire subsets can be classified in the following way: $\M_0(X)$ (~$\A_0(X)$~) is the class of all functionally closed (functionally open) subsets of the space $X$. For every countable ordinal $\alpha$,
$\M_\alpha(X)$ (~$\A_\alpha(X)$~) is the family of subsets of $X$, which can be presented as countable intersections (unions) of sets from
$\bigcup_{\xi<\alpha}\A_\xi(X)$ (from $\bigcup_{\xi<\alpha}\M_\xi(X)$~).

For a Borel subset $A$ of a Tychonoff space $X$ let us define a function $\dot \chi_A:\Pt(X)\to [0,1]$ with the help of the formula $\dot \chi_A(\mu)=\mu^*(A)$.

\begin{lemma}\label{l1.19} Let $X$ be a Tychonoff space, $A$ be
a subset of $X$, $\alpha$ be an even ordinal, $\xi$ be any ordinal, and $a\in\IR$.
Then
\begin{enumerate}
\item If $A\in \M_\xi(X)$, then
$\dot\chi_A^{-1}([a,\infty))=\{\mu\in\Pt(X)\mid \mu^*(A)\ge
a\}\in\M_\xi(\Pt(X))$.
\item If $A\in \F_\alpha(X)$, then
$\dot\chi_A^{-1}([a,\infty))=\{\mu\in\Pt(X)\mid \mu^*(A)\ge
a\}\in\F_\alpha(\Pt(X))$.
\item If $A\in \A_\xi(X)$, then
$\dot\chi_A^{-1}((a,\infty))=\{\mu\in\Pt(X)\mid \mu^*(A)>a\}\in\A_\xi(\Pt(X))$.
\item If $A\in \G_\alpha(X)$, then
$\dot\chi_A^{-1}((a,\infty))=\{\mu\in\Pt(X)\mid \mu^*(A)> a\}\in\G_\alpha(\Pt(X))$.
\end{enumerate}
\end{lemma}

\begin{proof} First, let us prove Statement (3) of this lemma. Let $U$ be a functionally open subset of $X$ and $\tilde U$ be a functionally open subset of $\beta X$ such that $\tilde U\cap X=U$.
It is easy to construct a sequence $\{f_n:\beta X\to [0,1]\}_{n=1}^\infty$ of continuous functions converging to the characteristic function $\chi_{\tilde U}:\beta X\to
[0,1]$
pointwise, such that $f_n|\beta X\bs \tilde U\equiv 0$, $f_n^{-1}(\{1\})\subset f^{-
1}_{n+1}(\{1\}),\; n\in\IN$,
and $\bigcup_{n=1}^\infty f_n^{-1}(\{1\})=\tilde U$. By 1.1,
$\mu^*(U)=\mu(\tilde U)$ for any $\mu\in\Pt(X)$. Consequently,
$$
\begin{aligned}
\dot\chi_U^{-1}((a,\infty))&=\{\mu\in\Pt(X)\mid \mu^*(U)>a\}=\\
&=\{\mu\in\Pt(X)\mid \mu(\tilde U)>a\}=\bigcup_{n=1}^\infty
\{\mu\in\Pt(X)\mid \mu(f_n)>a\}
\end{aligned}
$$ is a functionally open set in $\Pt(X)$.

If $Z$ is a functionally closed subset of $X$, then the set
$X\bs Z$ is functionally open. Choose   a functionally open subset $\tilde U$ of $\beta X$ such that $\tilde U\cap X=X\bs Z$,  and observe that the set $\{\mu\in\Pt(X)\mid \mu(\tilde U)>1-a\}$ is functionally open in $\Pt(X)$. Then
$$
\begin{aligned}
\dot\chi_Z([a,\infty ))&=\{\mu\in\Pt(X)\mid \mu^*(Z)\ge
a\}=\\
&=\{\mu\in\Pt(X)\mid \mu(\beta X\bs\tilde U)\ge a\}=\{\mu\in\Pt(X)\mid
\mu(\tilde U)\le 1-a\}
\end{aligned},
$$ being the complement to a functionally open set
$\{\mu\in\Pt(X)\mid\mu(\tilde U)>1-a\}$, is functionally closed in $\Pt(X)$.

Let us now show that for any closed set $F\subset X$, the set
$\dot\chi_F^{-1}([a,\infty ))=\{\mu\in\Pt(X)\mid \mu^*(F)\ge a\}$ is closed in $\Pt(X)$. Let $\bar F$ be the closure of the set $F$ in $\beta X$. Then (see \cite{4}) the set $\{\mu\in\Pt(\beta X)\mid \mu(\bar F)\ge a\}$ is closed in
$P(\beta X)$. By Lemma~\ref{l1.1}, $\mu^*(F)=\mu(\bar F)$ for each measure $\mu\in\Pt(X)$. Therefore, the set $$\Pt(X)\cap \{\mu\in
P(\beta X)\mid \mu(\bar F)\ge a\}=\{\mu\in\Pt(X)\mid \mu^*(F)\ge
a\}=\dot\chi^{-1}_F([a,\infty ))$$ is closed in $\Pt(X)$.
Switching to complements, we prove that for any open set $G$ in $X$ the set
$\dot \chi_G^{-1}((a,\infty))=\{\mu\in\Pt(X)\mid\mu^*(G)> a\}$ is open in
$\Pt(X)$.

Thus, for $\xi=\alpha=0$ the lemma is proved.

Now let $\xi$ be an ordinal. Let us assume that for every
$a\in\IR$ and $A\in\M_{\xi'}(X)$, where $\xi'<\xi$, it has been proved that
$\dot\chi_A^{-1}([a,\infty))\in\M_{\xi'}(\Pt(X))$. Let $A\in\M_\xi(X)$.
Then $A=\bigcap_{n=1}^\infty\bigcup_{m=1}^\infty A_n^m$, where $A_n^m\in
\bigcup_{\xi'<\xi}\M_{\xi'}(X)$. Without loss of generality, suppose that for every $n\in\IN$, $A_n^1\subset A_n^2\subset\dots $ and
$\bigcup_{m=1}^\infty A_1^m\supset\bigcup_{m=1}^\infty
A_2^m\supset\dots\,$.
 One can easily observe that
$\dot\chi_A^{-1}([a,\infty))=\{\mu\in\Pt(X)\mid\mu^*(A)\ge
a\}=\bigcap_{n=1}^\infty\bigcup_{m=1}^\infty\{\mu\in\Pt(X)\mid
\mu^*(A_n^m)\ge a-\frac1n\}$. By the induction hypothesis for every $n,m\in\IN$ \ $\{\mu\in\Pt(X)\mid \mu^*(A_n^m)\ge
a-\frac1n\}\in\M_{\xi'}(\Pt(X))$, where $\xi'<\xi$. Therefore,
$\dot\chi_A^{-1}([a,\infty))\in \M_\xi(\Pt(X))$.

In a similar fashion for any even ordinal
$\alpha$ we can prove that
$A\in\F_\alpha(X)$ implies $\dot\chi_A^{-1}([a,\infty))\in\F_\alpha(\Pt(X))$.

If $A\in\A_\xi(X)$, then
$$
\begin{aligned}
\Pt(X)\bs\dot\chi_A^{-1}((a,\infty
))&=\{\mu\in\Pt(X)\mid \mu^*(A)\notin (a,\infty)\}=\{\mu\in\Pt(X)\mid
\mu^*(A)\le a\}=\\
&=\{\mu\in\Pt(X)\mid \mu^*(X\bs A)\ge 1-a\}=\dot\chi_{X\bs
A}^{-1}([1-a,\infty)).
\end{aligned}.
$$ Since $A\in \A_\xi(X)$, $X\bs
A\in\M_\xi(X)$, and thus $\dot\chi^{-1}_{X\bs
A}([1-a,\infty))\in\M_\xi(\Pt(X))$ and
$\dot\chi_A^{-1}((a,\infty))=\Pt(X)\bs\dot\chi^{-1}_{X\bs
A}([1-a,\infty))\in\A_\xi(X)$.

In a similar fashion we show that for any even ordinal  $\alpha$ and
$A\in\G_\alpha(X)$ we get $\dot\chi_A^{-1}((a,\infty))\in\G_\alpha(\Pt(X))$. The lemma is proved.
\end{proof}

\begin{corollary}\label{c1.20} The functor $\Pt$ preserves \v Cech-complete spaces.
\end{corollary}

\begin{proof} Let $X$ be a \v Cech-complete Tychonoff space. Then $X=\bigcap_{n=1}^\infty U_n$ is a $G_\delta $-set in $\beta X$
(here $U_n\subset \beta X$, $n\in\IN$, are open sets in $\beta  X$). Then $$\Pt(X)=\{\mu\in P(\beta
X)\mid\mu(X)=1\}=\bigcup_{n=1}^\infty \{\mu\in P(\beta X)\mid
\mu(U_n)>1-\frac1n\}.$$ By Lemma 1.18 the sets $\{\mu\in P(\beta
X)\mid\mu(U_n)>1-\frac1n\}$ are open in $P(\beta X)$, which means that $\Pt(X)$ is a $G_\delta $-set in $P(\beta X)$. By \cite[3.9.1]{6}, the space
$\Pt(X)$ is \v Cech-complete.
\end{proof}

\begin{corollary}\label{c1.21} If $A$ is a Baire subset of a Tychonoff space $X$, then the function $\dot\chi_A:\Pt(X)\to[0,1]$ is measurable with respect to the $\sigma $-algebra of Baire subsets of $\Pt(X)$.
\end{corollary}

\begin{theorem}\label{t1.22} The functor $\Pt$ preserves Baire subsets. Moreover, for any ordinal number $\xi$, if $A\in\M_\xi(X)$, then
$\Pt(A)\in\M_\xi(\Pt(X))$; for any even ordinal $\alpha$, if $A\in
\F_\alpha(X)$, then $\Pt(A)\in\F_\alpha(\Pt(X))$.
\end{theorem}

Let $X$ be a metrizable compact space. By $\P(X)$ we denote the family of projective subsets of $X$, i.e. the smallest family containing the class $\mathcal B(X)$ of all Borel subsets of $X$ and satisfying the following conditions:
\begin{itemize}
\item[(1)] For any continuous map $f:A\to X$ of the set
$A\in\P(X)$ the image $f(A)$ belongs to the class $\P(X)$;
\item [(2)] for any set $A\in\P(X)$ its complement $X\bs A$
belongs to the set $\P(X)$.
\end{itemize}
The family $\P(X)$ can be presented as $\P(X)=\bigcup_{n=0}^\infty
\P_n(X)$, where $\P_0(X)=\mathcal B(X)$ is the family of Borel sets  of $X$;
projective sets of the class  $\P_{2n+1}(X)$ are continuous images of sets from the class $\P_{2n}(X)$; projective sets of the class $\P_{2n}(X)$ are complements to projective sets of the class $\P_{2n-1}(X)$. Given this, for any $n\ge 0$
$\P_{2n+1}(X)\subset\P_{2n+3}(X)\cap\P_{2n+4}(X)$ \cite[\S 38]{12}. Projective sets from classes $\P_1(X)$ и $\P_2(X)$ have specific names: they are called, respectively, analytic and coanalytic.

\begin{theorem}\label{t1.23} The functor $\Pt$ preserves projective subsets of metrizable compacta. Moreover, for any metrizable compact space $X$ and any $n\ge 1$, if
$A\in\P_{2n-1}(X)$, then $\Pt(A)\in\P_{2n+2}(P(X))$.
\end{theorem}

\begin{proof} Let $X$ be a metrizable compact space. By  $\exp(X)$ we denote the hyperspace of non-empty closed subsets of $X$ endowed with the Vietoris topology \cite{9}. Let us note that for any $0\le a\le 1$ the set
$R(a)=\{(\mu,K)\in P(X)\times \exp(X)\mid \mu(K)\ge a\}$ is closed in
$P(X)\times \exp(X)$. Indeed, let $(\mu,K)\in P(X)\times \exp(X)$ be a limit point of the set $R(a)$. By the definition,
$\mu(K)=\inf\{\mu(f)\mid f\in C(X),\; f\ge 0,\; f|K\equiv 1\}$. Let
$f:X\to [0,1]$ be any function with $f|K\equiv 1$. Then for any $\e>0$
the set $$\langle f^{-1}(1-\frac\e2,1]\rangle=\{B\in\exp(X)\mid B\subset
f^{-1}(1-\frac \e2,1]\}$$ is open in $\exp(X)$. Since $(\mu,K)$ is the limit point of the set $R(a)$, there exists a pair $(\eta, C)\in
R(a)$ such that $|\eta(f)-\mu(f)|<\frac \e2$ and $C\subset f^{-1}(1-\frac
\e2,1]$. Then $\mu(f)>\eta(f)-\frac \e2=\int_Xf\, d\,\eta-\frac
\e2\ge\int_Cf\, d\eta-\frac \e2>(1-\frac \e2)\eta(C)-\frac \e2\ge (1-\frac
\e2)a-\frac \e2\ge a-\e$. Since $\e>0$ is arbitrary, $\mu(f)\ge
a$, which implies that $\mu(K)\ge a$ and $(\mu,K)\in R(a)$.

Now let $n\ge 1$ and $A\in\P_{2n-1}(X)$. Then
$$
\begin{aligned} P(X)&\bs\Pt(A)=\{\mu\in
P(X)\mid\mu^*(A)<1\}=\{\mu\in P(X)\mid \mu_*(X\bs A)>0\}=\\
&=\bigcup_{m=1}^\infty \{\mu\in P(X)\mid \mu(K)\ge \tfrac
1m\mbox{ for some compact $K\subset X\bs A$}\}=\bigcup_{m=1}^\infty \pr_1(G_m),
\end{aligned}
$$ where
$$G_m=\{(\mu,K)\in P(X)\times\exp(X)\mid K\subset X{\bs}A,\; \mu(K)\ge
\tfrac 1m\}=R(\tfrac 1m)\cap (P(X)\times\exp(X{\bs}A)),$$ and
$\pr_1:P(X)\times\exp(X)\to P(X)$ is the projection onto the first factor. Since $A\in\P_{2n-1}(X)$, then $X\bs A\in\P_{2n}(X)$, and, according to \cite{13},
$\exp(X\bs A)\in\P_{2n}(\exp(X))$. Consequently,
$P(X)\bs\Pt(A)\in\P_{2n+1}(P(X))$ and
$\Pt(A)\in\P_{2n+2}(P(X))$. The theorem is proved.
\end{proof}

\begin{remark}\label{r1.24} Theorem \ref{t1.22} implies that $\Pt(A)\in\P_0(P(X))$ for any $A\in\P_0(X)$. Moreover, Theorem \ref{t2.32} implies that for any coanalytic subset $A\subset X$ \
$\Pt(A)\in\P_3(P(X))$. This follows from the fact that coanalytic subsets of metric compacta are measurable with respect to any measure and, therefore, for a coanalytic set  $A\subset X$ the equality $\Pt(A)=\HP(A)$ holds.
\end{remark}

Let us now recall the notion of monad, introduced by S. Eilenberg and T. Moore \cite{14}.

\begin{definition}\label{d1.25} A {\em monad} on a category ${\mathcal C}$ is a triple
${\Bbb T}=(T,\delta,\psi)$ consisting of a covariant functor $T:{\mathcal C}\to{\mathcal C}$
and natural transformations $\delta:Id\to T$ (identity) and $\psi:T^2\to T$ (multiplication) which satisfy the following conditions:
$$
\psi\circ T\delta=\psi\circ \delta T=\id_T\quad\text{и}\quad \psi\circ\psi
T=\psi\circ T\psi.
$$
A functor $T$ that can be included in a triple $\Bbb T$ is called {\em monadic} in the category ${\mathcal C}$.
\end{definition}

It is well-known \cite{2} that the functor $P$ is monadic in the category $\Comp$ of compacta. In fact, it can be included in the monad ${\Bbb P}=(P,\delta,\psi)$,
where $\delta$ is the Dirac transformation, and the component $\psi_X:P^2(X)\to
P(X)$ of multiplication $\psi$ is determined by the formula  $\psi_X(M)(f)=M(F_f)$ for
$f\in C(X)$, $M\in P^2(X)$, where $F_f:P(X)\to\IR$ is a continuous function such that $F_f(\mu)=\mu(f)$, $\mu\in P(X)$.

Our aim is to show that the functor $\Pt:\Tych\to\Tych$ can also be included in a monad. For this purpose, apparently, it is sufficient to show that for any Tychonoff space $X$ \ $\delta _{\beta X}(X)\subset \Pt(X)$ and
$\psi_{\beta X}(\Pt^2(X))\subset \Pt(X)$, where $\delta _{\beta X}$ and
$\psi_{\beta X}$ are components of natural transformations included in the triple ${\Bbb P}=(P,\delta ,\psi)$ (here and further in the text the symbol $\Pt^2$ denotes the composition of functors $\Pt\circ\Pt$). The first inclusion
$\delta _{\beta X}(X)\subset \Pt(X)$  follows from Theorem 1.14. In order to prove the second inclusion,  let us assume that $M\in \Pt^2(X)\subset
P^2(\beta X)$, i.e. $M^*(\Pt(X))=1$. Let $\{\varphi_\alpha \}\subset
C(\beta X)$ be a monotonically decreasing net of continuous functions on
$\beta X$ converging to zero on the set $X\subset \beta X$ pointwise.
According to \cite{4}, $\psi_{\beta X}(M)\in\Pt(X)$ provided we show that the real-valued net $\{\psi_{\beta X}(M)(\varphi_\alpha)\}$ converges to zero.
Observe that $\{F_{\varphi_\alpha }:P(\beta X)\to\IR\}$ is a monotonically decreasing net of continuous functions on $P(\beta X)$, converging to zero on the set $\Pt(X)\subset P(\beta X)$ pointwise. Since $M$ is a $\tau$-smooth measure on $\Pt(X)$, by \cite{4} we have that,
$\{M(F_{\varphi_\alpha })\}\to 0$. But $\psi_{\beta X}(M)(\varphi_\alpha)
=M(F_{\varphi_\alpha })$ for every $\alpha $. Therefore, the net
$\{\psi_{\beta X}(M)(\varphi_\alpha )\}$ converges to zero, i.e.
$\psi_{\beta X}(\Pt^2(X))\subset \Pt(X)$. Let us choose $\delta _X=\delta
_{\beta X}|X:X\to\Pt(X)$ and $\psi_X=\psi_{\beta
X}|\Pt^2(X):\Pt^2(X)\to\Pt(X)$. One can easily see that ${\Bbb
P}_\tau=(\Pt,\delta ,\psi)$ is a monad on the category $\Tych$. Thus, we have proved

\begin{theorem}\label{t1.26} The functor $\Pt:Tych\to\Tych$ is a monad on the category $\Tych$ extending the monad $P:\Comp\to\Comp$.
\end{theorem}

\begin{lemma}\label{l1.27} For any Tychonoff space $X$ and its subset $Y$ the equality $\psi_X^{-1}(\Pt(Y))=\Pt^2(Y)$ holds.
\end{lemma}

\begin{proof} The inclusion $\psi_X(\Pt^2(Y))\subset \Pt(Y)$ follows from the fact that the functor $\Pt$ is a monad. Let us now prove the inverse inclusion. Let
$M\in P^2(\beta X)$ be  a measure such that $\psi_{\beta X}(M)\notin \Pt(Y)$.
Then there exists a compact $K\subset \beta Y\bs Y$ such that $\psi_{\beta
X}(M)(K)>0$. Consider a family of functions
$\Phi=\{f\in C(\beta X)\mid 0\le f\le 1,\; f|K\equiv 1\}$, equipped with a natural partial order $\le$. The family $\Phi$ is downward-directed, i.e. for any functions $f,g\in\Phi$ \ $\min(f,g)\in\Phi$, and, if treated as a net, it converges to zero on the set $Y$ pointwise. Then the net $\{F_f:P(\beta X)\to\IR\}_{f\in \Phi}$
monotonically decreases and converges to zero on the set $\Pt(Y)$. If the measure
$M$ would belong to the set $\Pt^2(Y)$, the net
$\{M(F_f)\}_{f\in\Phi}$ would converge to zero. But $M(F_f)=\psi_{\beta
X}(M)(f)\ge \psi_{\beta X}(M)(K)>0$. This contradiction shows that $M\notin \Pt^2(Y)$. Thus $\psi_X^{-1}(\Pt(Y))=\Pt^2(Y)$.
\end{proof}

\begin{corollary}\label{c1.28} For every Tychonoff space $X$ the component $\psi_X:\Pt^2(X)\to\Pt(X)$ of multiplication is an open and perfect map.
\end{corollary}

\noindent{\it The proof} follows from Lemma 1.27 and from the fact that
$\psi_{\beta X} :P^2(\beta X)\to P(\beta X)$ is an open mapping of compacta \cite[7.8]{2}, or \cite{15}.

Let us recall that a map $p:X\to Y$ between topological spaces is called
{\it soft} for the class of metric spaces if for any metric space $A$, its closed subset $B\subset A$ and maps $g:A\to Y$ and $f:B\to X$ such that $p\circ g=g|B$ there exists a map $F:A\to X$ such that $F|B=f$ and $p\circ F=g$.

\begin{theorem}\label{t1.29} For every metric space $X$ the map  $\psi_X:\Pt^2(X)\to
\Pt(X)$ is soft for the class of metric spaces.
\end{theorem}

{The proof} follows from Corollary \ref{c1.28}, metrizability of the space
$\Pt^2(X)$ (see Theorem \ref{t1.17}), Michael's selection theorems \cite[\S1.4 and Ex. 1.4.2]{16} and the fact that the preimage
$\psi_X^{-1}(\mu)\subset\Pt^2(X)$ of any measure $\mu\in\Pt(X)$ is a convex compact in $\Pt^2(X)$.
\vskip5pt

We will say that the map $p:E\to B$ is homeomorphic to a trivial $Q$-fibration if there exists a homeomorphism $f:E\to B\times Q$ such that
$\pr_B\circ f=p$, where $\pr_B:B\times Q\to B$ is a natural projection.
Here $Q=[-1,1]^\omega$ stands for the Hilbert cube.

\begin{theorem}\label{t1.30} For a metrizable separable space
$X$ that contains more than one point, the map $\psi_X|\Pt^2(X)\bs
\delta^2(X):\Pt^2(X)\bs \delta^2(X)\to
\Pt(X)\bs\delta(X)$ is homeomorphic to a trivial $Q$-fibration.
\end{theorem}

\begin{proof} Let $c\, X$ be a metric compactification of the space $X\ne\{*\}$. In \cite{15} it is proved that the map
$\psi_X|P^2(c\, X)\bs \delta^2(c\, X):P^2(c\, X)\bs
\delta^2(c\, X)\to P(c\, X)\bs \delta(c\, X)$ is a trivial $Q$-fibration. Now the theorem follows from Lemma \ref{l1.27}.
\end{proof}

\section{Categorial properties of the functor $\HP$}

In this section we shall investigate categorial properties of the functor
$\HP:\Tych\to\Tych$ of Radon probability measures.

As it was mentioned in the introduction, there exist two equivalent approaches to defining the space $\HP(X)$, where $X$ is a Tychonoff space. The first is via embedding in compact spaces:
$\HP(X)=\{\mu\in P(\beta X)\mid \mu_*(X)=1\}\subset P(\beta X)$. In the second approach, $\HP(X)$ is defined as the space of all Radon probability measures on $X$. Further in the text, depending on the situation, without any specific caveats we will use either the first or the second approach to the description of the space $\HP(X)$.

\begin{theorem}\label{t2.1} The functor $\HP$ preserves the class of injective maps.
\end{theorem}

\begin{proof} Let $f:X\to Y$ be an injective map and
$\mu_1,\mu_2\in \HP(X)$, $\mu_1\ne \mu_2$. Then $\mu_1(A)\ne\mu_2(A)$ for some Borel set $A\subset X$. Since the measures
$\mu_1,\mu_2$ are Radon, there exists a compact $K\subset A$ such that
$\mu_1(K)\ne \mu_2(K)$. Then $f(K)$ is a compact in $Y$ and
$$\HP(f)(\mu_1)(f(K))=\mu_1(f^{-1}(f(K)))=\mu_1(K)\ne
\mu_2(K)=\HP(f)(\mu_2)(f(K)),$$ i.e. $\HP(f)(\mu_1)\ne \HP(f)(\mu_2)$.
The theorem is proved.
\end{proof}

\begin{theorem}\label{t2.2} The functor $\HP$ preserves the class of perfect maps.
\end{theorem}

\begin{proof} Let $f:X\to Y$ be a perfect map of Tychonoff spaces.Then the extension $\beta f:\beta X\to \beta Y$ of the map $f$ has the following property: $\beta f(\beta X\bs X)\subset
\beta Y\bs Y$ \cite[3.7.15]{6}.
We will show that $P(\beta f)(P(\beta X)\bs \HP(X))\subset P(\beta Y)\bs
\HP(Y)$. Let $\mu\in P(\beta X)$ and $P(\beta
f)(\mu)\in \HP(Y)$. Then for any $\e>0$ there exists a compact
$K\subset Y\subset \beta Y$ such that $P(\beta f)(\mu)(K)>1-\e$.
Since $f$ is a proper map, $(\beta f)^{-1}(K)=f^{-1}(K)$
is a compact in $X$ (see \cite[3.7.2]{6}). Then $P(\beta f)(\mu)(K)=\mu((\beta
f)^{-1}(K))=\mu(f^{-1}(K))>1-\e$. Therefore, $\mu\in \HP(X)$ and
$P(\beta f)(P(\beta X)\bs \HP(X))\subset P(\beta Y)\bs \HP(Y)$.
Since $P(\beta f):P(\beta X)\to P(\beta Y)$ is a map between compacta, the last inclusion implies that the map
$\HP(f)=P(\beta f)|\HP_\beta(X):\HP(X)\to\HP(Y)$ is perfect. The theorem is proved.
\end{proof}

\begin{corollary}\label{c2.3} The functor $\HP$ preserves the class of closed embeddings.
\end{corollary}

Since $\HP$ is a subfunctor of the functor $P_\tau$, Theorem \ref{t1.4} implies

\begin{theorem}\label{t2.4} The functor $\HP$ preserves the class of topological embeddings.
\end{theorem}

Thus, the functor $\HP$ preserves the class of injective maps and the class of (closed) embeddings. But in the case of surjective maps it is not so straightforward.

We say that a map $f:X\to Y$ {\it has the property of Borel selection} if there exists a map $s:Y\to X$ (not necessarily continuous) such that
$f\circ s=\id_Y$, and for any open set $U\subset X$ \
$s^{-1}(U)$ is a Borel subset of the space $Y$.

The map $f:X\to Y$ {\it has local Borel selections}, if for every open set $U\subset X$ there exists a Borel selection $s:Y\to X$ of the map $f$ such that $s(f(U))\subset U$.

\begin{example}\label{e2.5} Let $p:{\frak c}\to [0,1]$ be a bijective map of a discrete space $\frak c$ onto an interval. Then the map
$\HP(p):\HP({\frak c})\to \HP([0,1])$ is not surjective (the Lebesque measure on [0,1] does not have a preimage).
\end{example}

At the same time, the following is also true:

\begin{proposition}\label{p2.6} Let $f:X\to Y$ be a map between separable metric spaces that has a Borel selection. Then the map $\HP(f):\HP(X)\to \HP(Y)$ is surjective.
\end{proposition}

\begin{proof} Let $s:Y\to X$ be a Borel selection of the map
$f$. For every measure $\mu\in\HP(Y)$ let us choose $\eta$, a countably additive probability measure on $X$ defined by the condition $\eta(A)=\mu(f(A\cap
s(Y)))$ for every Borel set $A\subset X$.

Let us show that $\eta$ is a Radon measure. It is sufficient to show that for any $\e>0$ there exists a compact $K\subset X$ such that $\eta(K)>1-\e$.
Let us fix $\e>0$. Since the map $s:Y\to X$ is Borel measurable, the Lusin theorem \cite[2.3.5]{7} implies that there exists a closed subset $C\subset Y$ such that $\mu(C)>1-\e/2$ and the map
$s|C:C\to X$ is continuous. Since the measure $\mu$ on $Y$ is Radon, there exists such a compact $K\subset C$ that $\mu(C\bs K)<\frac\e2$. Then
$s(K)\subset X$ is compact. Furthermore, $\eta(s(K))=\mu(f(s(K)\cap
s(Y))=\mu(f(s(K)))=\mu(K)>1-\e$. Thus the measure $\eta$ on $X$ is Radon and $\HP(f)(\eta)=\mu$. The theorem is proved.
\end{proof}

\begin{corollary}\label{c2.7} Let $f:X\to Y$ be a bijective continuous mapping of a separable Borel space $X$ onto a metric space $Y$. Then the map $\HP(f):\HP(X)\to \HP(Y)$ is bijective.
\end{corollary}

\begin{proof}The injectivity of the map $\HP(f)$ follows from Theorem \ref{t2.1}. The surjectivity of $\HP(f)$ follows from Proposition \ref{p2.6}, since the map $f^{-1}:Y\to X$ is Borel measurable \cite[\S 39, IV]{12}.
\end{proof}

The condition that a Borel selection should exist is crucial here.
Indeed, let us consider

\begin{example}\label{e2.8} Let $Z\subset [0,1]$ be a subset of the interval with inner Lebesque measure $\lambda_*(Z)=0$ and outer measure $\lambda^*(Z)=1$.
Let $X=Z\times\{0\}\cup([0,1]\bs Z)\times \{1\}$ и $f:X\to [0,1]$ be a projection onto the first factor. Then the Lebesque measure $\lambda$ on $[0,1]$ does not have a preimage under the map $\HP(f):\HP(X)\to P([0,1])$.
\end{example}

At the same time, the functor $\HP$ preserves a feature of maps which implies surjectivity in the compact case.

\begin{proposition}\label{p2.9} Let $f:X\to Y$ be a map such that the image $f(X)$ is dense in $Y$. Then the image $\HP(f)(\HP(X))$ is dense in $\HP(Y)$.
\end{proposition}

\noindent{\it The proof\/} is similar to the proof of Proposition 1.6.

\begin{proposition}\label{p2.10} Let $f:X\to Y$ be an open map between separable metric spaces that has local Borel selections. Then the map $\HP(f):\HP(X)\to \HP(Y)$ is surjective and open.
\end{proposition}

\begin{proof} By Proposition \ref{p2.6}, the map $\HP(f)$ is surjective. Let us now show that it is open. According to \cite[II, \S1]{4}, the system of sets
$\N^*(\mu_0,U_1,\dots,U_n,\e)=\{\mu\in\HP(X)\mid
\mu(U_i)-\mu_0(U_i)>-\e,\; 1\le i\le n\}$, where $\e>0$, $\mu_0\in\HP(X)$ and
$U_1,\dots,U_n$ are open sets in $X$, forms a base for the topology on $\HP(X)$. Let us fix a base set $\N^*(\mu_0,U_1,\dots,U_n,\e)$ and show that its image $\HP(f)(\N^*(\mu_0,U_1,\dots,U_n,\e))$ is a neighborhood of the measure $\eta_0=\HP(f)(\mu_0)\in\HP(Y)$. To achieve this, we will first find a base neighborhood $\N^*(\mu_0,V_1,\dots,V_m,\e')\subset
\N^*(\mu_0,U_1,\dots,U_n,\e)$ such that $V_i$, $1\le i\le m$, are pairwise disjoint open subsets of $X$.

By $\In$ we will denote the $n$-element set $\{1,\dots,n\}$. We will equip the set $\exp(\In)$ of all non-empty subsets of $\In$ with a linear order such that for any $A,B\subset\In$, if $A\supset B$, then
$A\le B$ (see \cite[\S2.4, Theorem 4]{17}). Let us note that $|\exp(\In)|<2^n$.
Fix $\e'=\e/2^{n+1}$. For every $A\subset \In$ we choose $U_A=\bigcap_{i\in A}U_i$.
By induction, for every $A\subset \In$ find an open set
$V_A\subset X$ such that $\bar V_A\subset U_A\bs \bigcup_{B<A}\bar V_B$ and
$\mu_0(V_A)>\mu_0(U_A\bs \bigcup_{B<A}\bar V_B)-\e'$. One can easily see that for any $A\subset \In$ \ $\mu_0(\bar V_A\bs V_A)<\e'$ and, therefore,
$\mu_0(U_A)<\mu_0(V_A)+\sum_{B<A}\mu(\bar V_B)+\e'<\sum_{B<
A}\mu_0(V_B)+2^n\e'$. Also, it is obvious that  $\bar
V_A\cap \bar V_B=\emptyset$ for any $A\ne B$. We claim that
$\N^*(\mu_0,\{V_A:A\subset \In\},\e')\subset\N^*(\mu_0,U_1,\dots,U_n,\e)$.
Indeed, if $\mu\in\N^*(\mu_0,\{V_A:A\subset\In\},\e')$, then
$\mu(U_i)=\sum_{A\ni i}\mu(\bar V_A)+\mu(U_i\bs \bigcup_{A\ni i}\bar V_A)\ge
\sum_{A\ni i}\mu(\bar V_A)>\sum_{A\ni i}(\mu_0(V_A)-\e')>\sum_{A\ni
i}\mu_0(V_A)-2^n\e'>\mu_0(U_i)-2^{n+1}\e'=\mu_0(U_i)-\e$, $1\le i\le n$.
That is, $\mu\in\N^*(\mu_0,U_1,\dots,U_n,\e)$.

Let us present the set $\N^*(\mu_0,\{V_A:A\subset\In\},\e')$ as
$\N^*(\mu_0,V_1,\dots,V_m,\e')$, where $m=|\exp(\In)|$. By ${\mathbf m}$ we will denote the $m$-element set $\{1,\dots,m\}$. Let us equip the set
$\exp({\mathbf m})$ with a linear order such that for any $A,B\subset{\mathbf m}$,
if $A\supset B$, then $A\le B$. For every $A\subset{\mathbf m}$ fix
$W_A'=\bigcap_{i\in A}f(V_i)$. Since $f$ is an open map, the sets $W_A'\subset Y$ are open. Fix $\delta=\e'/2^{m+1}$.
By induction, for every $A\subset{\mathbf m}$ find an open set
$W_A\subset Y$ such that $\bar W_A\subset W'_A\bs \bigcup_{B<A}\bar W_B$ and
$\eta_0(W_A)>\eta_0(W'_A\bs \bigcup_{B<A}\bar W_B)-\delta$. One can easily observe that for any distinct $A,B\subset{\mathbf m}$ the sets $\bar W_A$ и $\bar W_B$
are disjoint. Furthermore, for any $A\subset {\mathbf m}$,
$\eta_0((W_A'\bs\bigcup_{B<A}W'_B)\bs W_A)<\delta$. We claim that
$\N^*(\eta_0,\{W_A:A\subset{\mathbf m}\},\delta)\subset \HP(f)
\N^*(\mu_0,V_1,\dots,V_m,\e'))$.
Indeed, let $\eta\in \N^*(\eta_0,\{W_A:A\subset{\mathbf m}\},\delta)$.
For every $A\subset{\mathbf m}$ and every $i\in A$, fix a Borel selection $s_{A,i}:Y\to X$ of the map $f$ such that $s_{A,i}(W_A)\subset V_i$.
Let $\alpha^A_i$, $i\in A$ be non-negative numbers such  that for any
$A\subset {\mathbf m}$ we get $\sum_{i\in A}\alpha_i^A=1$ and $\alpha^A_i\eta_0(W_A)\ge
\mu_0(f^{-1}(W_A)\cap V_i))$. Fix an arbitrary Borel selection
$s_0:Y\to X$ of the map $f$. Let $\mu$ be a measure on $X$ such that for any Borel set $C\subset X$
$$
\mu(C)=\eta(f(s_0(Y\bs \bigcup_{A\subset{\mathbf m}}W_A))\cap C)+\sum_{A\subset
{\mathbf m}}\sum_{i\in A}\alpha^A_i\eta(f(s_{A,i}(W_A)\cap C)).
$$
Similarly to the proof  of Proposition~\ref{p2.6}, it can be shown that $\mu$ is a Radon probability on $X$, i.e. $\mu\in\HP(X)$, and
$\HP(f)(\mu)=\eta$. We will show that $\mu\in \N^*(\mu_0,V_1,\dots,V_m,\e')$.
Indeed, $\mu(V_i)\ge \sum_{A\ni i}\alpha^A_i\eta(f(s_{A,i}(W_A)\cap
V_i))=\sum_{A\ni i}\alpha^A_i\eta(W_A)>\sum_{A\ni
i}\alpha^A_i(\eta_0(W_A)-\delta)>\sum_{A\ni
i}\alpha^A_i\eta_0(W_A)-2^m\delta\ge\sum_{A\ni i}\mu_0(f^{-1}(W_A)\cap
V_i)-2^m\delta=\mu_0(f^{-1}(\bigcup_{A\ni i}W_A)\cap
V_i)-2^m\delta=\mu_0(f^{-1}(\bigcup_{A\ni i}W'_A)\cap
V_i)-\mu_0(f^{-1}(\bigcup_{A\ni i}W'_A\bs\bigcup_{A\ni
i}W_A))-2^m\delta=\mu_0(V_i)-2^m\delta-\eta_0(\bigcup_{A\ni i}W'_A\bs
\bigcup_{A\ni i}W_A)$. What remains to be done is to assess the value $\eta_0(\bigcup_{A\ni
i}W'_A\bs \bigcup_{A\ni i}W_A)$. By the definition of sets $W_A'$,
$\bigcup_{A\ni i}W'_A=W'_i$. Then $\eta_0(\bigcup_{A\ni
i}W'_A\bs \bigcup_{A\ni i}W_A)=\eta_0(W'_i\bs  \bigcup_{A\ni
i}W_A)=\eta_0(\bigcup_{A\ni i}(W'_A\bs\bigcup_{B<A}W'_B)\bs \bigcup_{A\ni
i}W_A)\le \sum_{A\ni i}\eta_0((W'_A\bs \bigcup_{B<A}W'_B)\bs
W_A)<2^m\delta$. Finally, we get that
$\mu(V_i)>\mu_0(V_i)-2^{m+1}\delta=\mu_0(V_i)-\e'$.
\end{proof}

\begin{proposition}\label{p2.11} Let $f:X\to Y$ be a map between Tychonoff spaces, If $\HP(f):\HP(X)\to \HP(Y)$ is an open map, then the map $f$ is also open.
\end{proposition}

\noindent{\it The proof\/} literally repeats the proof of Proposition 4.1
\cite{18}.

\begin{remark}\label{r2.12} Under the assumption of the continuum hypothesis ( $\aleph_1=\frak
c$ ), the condition of separability in Propositions \ref{p2.6}, \ref{p2.10} can be omitted. This follows from \cite[\S31, X,8]{12} and the fact that the support of any Radon measure is separable.
\end{remark}

\begin{question}\label{q2.13}\footnote{${}^3$ This question was answered affirmatively in \cite{BK}.}
 Let $f:X\to Y$ be an open surjective map between separable Borel spaces. Will the map $\HP(f)$ be open?
\end{question}

In \cite[\S 3]{8} this question will be answered in the affirmative in the case when $X$  is metrizable by a complete metric.

Let $A$ be a subset of a Tychonoff space $X$. Since the functor
$\HP$ preserves embeddings, we will threat the space $\HP(A)$ as a subset of the space $\HP(X)$.

\begin{theorem}\label{t2.14} The functor $\HP$ preserves preimages, i.e. for any map $f:X\to Y$ between Tychonoff spaces and any subset
$A\subset Y$ we get $\HP(f)^{-1}(\HP(A))=\HP(f^{-1}(A))$.
\end{theorem}

\begin{proof} The inclusion $\HP(f^{-1}(A))\subset
\HP(f)^{-1}(\HP(A))$ is simple. We will show that $\HP(f)^{-1}(\HP(A))\subset
\HP(f^{-1}(A))$. Let $\mu\in\HP(X)$ be a measure satisfying the condition
$\HP(f)(\mu)\in\HP(A)$. Let us fix $\e>0$. Since
$\HP(f)(\mu)\in\HP(A)$, there exists a compact $K\subset A$ such that
$\HP(f)(\mu)(K)>1-\frac\e2$. The set $f^{-1}(K)\subset X $ is closed, with $\mu(f^{-1}(K))=\HP(f)(\mu)(K)>1-\frac\e2$. Since the measure $\mu$ is Radon,  there exists a compact $C\subset f^{-1}(K)$ such that
$\mu(f^{-1}(K)\bs C)<\frac\e2$. Therefore, $C\subset f^{-1}(A)$ and
$\mu(C)>1-\e$, т.е. $\mu\in\HP(f^{-1}(A))$. The theorem is proved.
\end{proof}

\begin{theorem}\label{t2.15} The functor $\HP$ preserves countable intersections, i.e. for any Tychonoff space $X$ and its subsets $X_n\subset X$,
$n\in\IN$, $\HP(\bigcap_{n\in\IN}X_n)=\bigcap_{n\in\IN}\HP(X_n)$.
\end{theorem}

\begin{proof} The inclusion $\HP(\bigcap_{n\in\IN}X_n)\subset
\bigcap_{n\in\IN}\HP(X_n)$ is obvious. Now let
$\mu\in\bigcap_{n\in\IN}\HP(X_n)$. We will show that
$\mu\in\HP(\bigcap_{n\in\IN}X_n)$. Fix $\e>0$. Since $\mu\in
\HP(X_n)$, $n\in\IN$, for every $n\in\IN$ there exists a compact
$K_n\subset X_n$ such that $\mu(K_n)>1-\e/{2^n}$. Let
$K=\bigcap_{n\in\IN}K_n$. One can easily check that $K\subset
\bigcap_{n\in\IN}X_n$ и $\mu(K)>1-\e$, i.e. $\mu\in
\HP(\bigcap_{n\in\IN}X_n)$. The theorem is proved.
\end{proof}

\begin{remark}\label{r2.16} Theorem \ref{t2.15} does not hold for an arbitrary number of indices. Indeed, let $X=[0,1]$ and $X_\alpha=[0,1]\bs\{\alpha\}$, where $\alpha\in[0,1]$.
Then for every $\alpha\in [0,1]$ the Lebesque measure $\lambda$ belongs to the set
$\HP(X_\alpha)$. But $\bigcap_{\alpha\in[0,1]}X_\alpha=\emptyset$. That is
$\HP(\bigcap_{\alpha\in[0,1]}X_\alpha)\ne \bigcap_{\alpha\in[0,1]}\HP(X_\alpha)$.
\end{remark}

\begin{lemma}\label{l2.17} Let $X$ be a Tychonoff space and $B\subset
X$ be a Borel subset of $X$. Then $\HP(B)=\Pt(B)\cap\HP(X)\subset
P(\beta X)$.
\end{lemma}

\begin{proof} The inclusion $\HP(B)\subset \Pt(B)\cap \HP(X)$
is obvious. Let $\mu\in\Pt(B)\cap\HP(X)$. Then $\mu^*(B)=1$ and
$\mu_*(X)=1$. Choose a Borel subset $\tilde B\subset \beta X$ such that $\tilde B\cap X=B$. Then $\mu(\tilde B)\ge
\mu^*(B)=1$. Since the measure $\mu$ is regular, for any $\e>0$
there exists a compact $K_1\subset \tilde B$ such that $\mu(\tilde B\bs
K_1)<\e/2$. By definition, $\mu_*(X)=1$ implies that there exists a compact subset $K_2\subset X\subset \beta X$ such that $\mu(K_2)>1-\e/2$. Then
$K=K_1\cap K_2\subset \tilde B\cap B=B$ is a compact subset of $B$ satisfying
$\mu(\beta X\bs K)\le \mu(\beta X\bs K_1)+\mu(\beta X\bs
K_2)<\e/2+\e/2=\e$, which, by the arbitrariness of $\e>0$ implies that $\mu_*(B)=1$
and $\mu\in\HP(B)$. The lemma is proved.
\end{proof}

\begin{theorem}\label{t2.18} The functor $\HP$ preserves intersections of closed subsets, i.e. for any Tychonoff space $X$ and its closed subsets $X_\alpha$, $\alpha\in A$, \ $\HP(\bigcap_{\alpha\in A}X_\alpha)=\bigcap_{\alpha\in
A}\HP(X_\alpha)$.
\end{theorem}

\begin{proof} Theorem \ref{t1.10} implies that
 $\Pt(\bigcap_{\alpha\in A}X_\alpha)=\bigcap_{\alpha\in A}\Pt(X_\alpha)$.
Then, by Lemma 2.17, $\HP(\cap_{\alpha \in A}X_\alpha
)=\Pt(\cap_{\alpha \in A}X_\alpha)\cap \HP(X)=\bigcap_{\alpha \in A}\Pt(X_\alpha
)\cap \HP(X)=\bigcap_{\alpha \in A}\HP(X_\alpha )$.
The theorem is proved.
\end{proof}

Now let us consider the question of continuity of the functor $\HP$. Let $A$ be a directed partially ordered set (which means that for any $\alpha,\beta\in A$ there exists a $\gamma\in A$ such that
$\gamma \ge \alpha$ и $\gamma\ge\beta$ ).

Let $\{X_\alpha,p_\alpha^\beta\}$ be an inverse system indexed by the set
$A$ and consisting of Tychonoff spaces. By $\invlim X_\alpha$ we denote the limit of that system, and by $p_\alpha:\invlim X_\alpha\to X_\alpha$,
$\alpha\in A$, -- the bonding maps.

The inverse system $\{X_\alpha,p_\alpha^\beta\}$ induces the inverse system
$\{\HP(X_\alpha),\HP(p_\alpha^\beta)\}$, whose limit is denoted by
$\invlim \HP(X_\alpha)$, and the limit projections by $\pr_\alpha:\invlim
\HP(X_\alpha)\to\HP(X_\alpha)$. The maps $\HP(p_\alpha):\HP(\invlim
X_\alpha)\to \HP(X_\alpha)$ induce a map $R:\HP(\invlim X_\alpha)\to \invlim
\HP(X_\alpha)$.

It is well-known that if all  $X_\alpha$ are compact, then the map $R$ is a homeomorphism. This follows from the continuity of the functor $P$ in the category of compacta \cite[VII.3.11]{9}.

\begin{theorem}\label{t2.19} The map $R:\HP(\invlim X_\alpha)\to \invlim
\HP(X_\alpha)$ is an embedding. If the limit projections
$p_\alpha:\invlim X_\alpha \to X_\alpha$ are dense, then the image $R(\HP(\invlim
X_\alpha))$ is dense $\invlim \HP(X_\alpha)$. If the index set $A$ is countable, then $R$ is a homeomorphism.
\end{theorem}

\begin{proof} The first two statements are proved in a similar fashion to corresponding statements of Theorem \ref{t1.11}.
Let us assume that the set $A$ is countable and show that the map
$R:\HP(\invlim X_\alpha)\to\invlim\HP(X_\alpha)$ is a homeomorphism. For this purpose it is sufficient to prove the surjectivity of the map $R$. Like in the proof of
1.11, let us embed the map $R$ in the homeomorphism $\bar R:P(\invlim \beta X_\alpha)\to \invlim
P(\beta X_\alpha)$.

Fix a thread $\{\mu_\alpha\}_{\alpha\in A}\in \invlim \HP(\beta X_\alpha)$. Let us show that $\mu=\bar R^{-1}(\{\mu_\alpha\}_{\alpha\in A})\in\HP(\invlim
X_\alpha)\subset \HP(\invlim \beta X_\alpha)$. Choose an $\e>0$. Fix a bijection
$\xi:A\to\IN$. For every $\alpha\in A$ choose a compact $K_\alpha\subset
X_\alpha$ satisfying $\mu_\alpha(K_\alpha)>1-\e\cdot 2^{-\xi(\alpha)}$. One can easily observe that the set
$K=\{(x_\alpha)_{\alpha\in A}\in\invlim X_\alpha\mid p_\alpha(x_\alpha)\in K_\alpha, \; \alpha\in A\}$
is compact. Furthermore, $\mu((\invlim X_\alpha)\bs K)\le \bigcup_{\alpha\in
A}\mu(p_\alpha^{-1}(X_\alpha\bs K_\alpha))=\bigcup_{\alpha\in A}\mu_\alpha(X_\alpha\bs
K_\alpha)\le\sum_{\alpha\in
A}\e\cdot 2^{-\xi(\alpha)}=\e$. Therefore, the map $R$ is surjective and the theorem is proved.
\end{proof}

Corollary \ref{c1.13} implies

\begin{proposition}\label{p2.20} The functor $\HP$ preserves homotopies, i.e. for any homotopy $H_t:X\to Y$ the homotopy $\HP(H_t):\HP(X)\to \HP(Y)$ is continuous as a map $\HP(H_{(\cdot)}):\HP(X)\times [0,1]\to \HP(Y)$.
\end{proposition}

Now we will consider the operation of tensor product of Radon probability measures. It is well-known (see \cite[VIII,\S 1]{9}) that given a family $\{X_\alpha \}_{\alpha \in A}$ of compacts, for any probability measures $\mu_\alpha \in P(X_\alpha )$, $\alpha \in A$,
there exists a unique measure $\oplu\limits_{\alpha \in A}\mu_\alpha \in
P(\prod_{\alpha \in A}X_\alpha )$ (which is called the tensor product of measures
$\mu_\alpha $ ) on the product $\prod_{\alpha \in A}X_\alpha $ satisfying the following condition: for any finite $B\subset A$ and any Borel sets $Y_\alpha \subset X_\alpha $, $\alpha \in A$, where
$Y_\alpha =X_\alpha $, if $\alpha \notin B$, the following holds:
$\oplu\limits_{\alpha \in A}\mu_\alpha(\prod_{\alpha \in A}Y_\alpha
)=\prod_{\alpha
\in A}\mu_\alpha (Y_\alpha )$.

\begin{proposition}\label{p2.21} Let $\{X_\alpha\}_{\alpha \in A}$ be a family of Tychonoff spaces, $\{c\,X_\alpha\}_{\alpha \in A}$ be a family of their compactifications and $\mu_\alpha \in\HP(X_\alpha )\subset P(c\,
X_\alpha )$, $\alpha \in A$, be a family of Radon probability measures. If the index set $A$ is at most countable, then
$\oplu\limits_{\alpha \in A}\mu_\alpha \in\HP(\prod_{\alpha \in
A}X_\alpha )\subset P(\prod_{\alpha \in A}c\,X_\alpha )$.
\end{proposition}

\begin{proof} Let us show that the measure $\oplu_{\alpha \in A}\mu_\alpha $
belongs to the set  $\HP(\prod_{\alpha \in A}X_\alpha)\subset
P(\prod_{\alpha\in
A}c\, X_\alpha)$. For this purpose fix an arbitrary $\e>0$. As the set $A$ is at most countable, there exists an injection $\xi:A\to\IN$. Since every measure $\mu_\alpha \in\HP(X_\alpha )$ is Radon, for every
$\alpha\in A$ there exists a compact $K_\alpha \subset X_\alpha \subset c\,
X_\alpha $ such that $\mu_\alpha (c\, X_\alpha \bs K_\alpha )<\e/2^{-\xi(\alpha
)}$. Then for the compact space $\prod_{\alpha \in A}K_\alpha \subset \prod_{\alpha\in
A}X_\alpha $ we get: $$(\oplu\limits_{\alpha\in
A}\mu_\alpha)(\prod_{\alpha\in
A}c\, X_\alpha \bs \prod_{\alpha\in A}K_\alpha)\le \sum_{\alpha\in A}\mu_\alpha(c\,
X_\alpha\bs K_\alpha)<\sum_{\alpha\in A}\e/2^{-\xi(\alpha)}\le \e.$$ Consequently,
$\oplu_{\alpha\in A}\mu_\alpha\in \HP(\prod_{\alpha\in A}X_\alpha)$.
\end{proof}

\begin{remark}\label{r2.22} Proposition \ref{p2.21} and well-known facts about the tensor product of probability measures on compact spaces implies that for any at most countable set of Radon probability measures
$\mu_\alpha\in\HP(X_\alpha)$, $\alpha\in A$, on Tychonoff spaces $X_\alpha$,
there exists a unique Radon probability measure
$\oplu\limits_{\alpha \in A}\mu_\alpha \in
\HP(\prod_{\alpha \in A}X_\alpha )$ (which is called {\em the tensor product of measures}
$\mu_\alpha $ ) such that for any Borel sets
$Y_\alpha \subset X_\alpha $, $\alpha \in A$,
we have the following equality:
$(\oplu\limits_{\alpha \in A}\mu_\alpha)(\prod_{\alpha \in A}Y_\alpha
)=\prod_{\alpha
\in A}\mu_\alpha (Y_\alpha )$.
\end{remark}

\begin{remark}\label{r2.23} Proposition \ref{p2.21} does not hold if the index set $A$ is uncountable. Indeed, if every measure $\mu_\alpha\in X_\alpha$,
$\alpha\in A$, has a non-compact support
$\supp_{cX_\alpha}(\mu_\alpha )\cap X_\alpha $, then it can be shown that the measure
$(\oplu\limits_{\alpha\in A}\mu_\alpha)(K)$ of any compact set
$K\subset \prod_{\alpha\in A}X_\alpha\subset \prod_{\alpha\in A}c\, X_\alpha$ is zero.
\end{remark}

For every Tychonoff space $X$ let us define a map $\delta_X:X\to \HP(X)$ assigning to each point $x\in X$  the Dirac measure
$\delta_X(x)$, concentrated at the point $x$.

Theorem \ref{t1.14} implies

\begin{theorem}\label{t2.24} The family $\delta=\{\delta_X\}$ defines a unique natural transformation of the identity functor $\Id:\Tych\to \Tych$ to the functor $\HP:\Tych\to\Tych$, whose  components
$\delta_X:X\to\HP(X)$ are a closed embeddings.
\end{theorem}

In a similar fashion to Theorem \ref{t1.15} we can prove

\begin{theorem}\label{t2.25} The functor $\HP$ preserves the density of Tychonoff spaces, i.e. $d(\HP(X))=d(X)$ for any infinite Tychonoff space $X$.
\end{theorem}

Theorem \ref{t1.16}---\ref{t1.22}, and also Lemma \ref{l2.17} imply

\begin{theorem}\label{t2.26} The functor $\HP$ preserves the weight of Tychonoff spaces, i.e.
  $w(\HP(X))=w(X)$ for any infinite Tychonoff space $X$.
\end{theorem}

\begin{theorem}\label{t2.27} The functor $\HP$ preserves the class of metrizable spaces.
\end{theorem}

\begin{proposition}\label{p2.28} The functor $\HP$ preserves \v Cech-complete spaces.
\end{proposition}

\begin{proposition}\label{p2.29} If $A$ is a Baire subset of a Tychonoff space $X$, then the function $\hat\chi_A:\HP(X)\to[0,1]$,
where $\hat\chi_A(\mu)=\mu(A)$, $\mu\in\HP(X)$, is measurable with respect to the $\sigma $-algebra of Baire subsets of $\HP(X)$.
\end{proposition}

\begin{theorem}\label{t2.30} The functor $\HP$ preserves Baire subsets. Moreover, for any ordinal number $\xi$, if $A\in\M_\xi(X)$, then
$\HP(A)\in\M_\xi(\HP(X))$; for every even ordinal number $\alpha$, if $A\in
\F_\alpha(X)$, then $\HP(A)\in\F_\alpha(\HP(X))$.
\end{theorem}

Theorems  \ref{t2.2} and \ref{t2.27} imply

\begin{theorem}\label{t2.31} The functor $\HP$ preserves the class of $p$-paracompact spaces.
\end{theorem}

\begin{theorem}\label{t2.32} The functor $\HP$ preserves projective subsets of metrizable compacta. Furthermore, for every $n\ge 0$, if
$A\in\P_{2n}(X)$, then $\HP(A)\in\P_{2n+1}(P(X))$.
\end{theorem}

\begin{proof} Let $X$ be a metric compact. For $n=0$
the statement of the theorem follows from Theorem \ref{t2.30}. By $\exp(X)$ we denote the hyperspace of non-empty closed subsets of $X$, equipped with the Vietoris topology.

Now let $n\ge 1$ and $A\in\P_{2n}(X)$. In this case $\HP(A)=\{\mu\in
P(X)\mid\mu_*(A)=1\}=\{\mu\in P(X)\mid$ for any $m\ge 1$ there exists a compact $K\subset A$ such that $\mu(K)\ge 1-\frac
1m\}=\bigcap_{m=1}^\infty \pr_1(E_m)$, where $E_m=\{(\mu,K)\in
P(X)\times\exp(X)\mid K\subset A$ and $\mu(K)\ge 1-\frac 1m\}$ and
$\pr_1:P(X)\times\exp(X)\to P(X)$, the projection onto the first factor.
Let $\exp(A)=\{C\in \exp(X)\mid C\subset A\}$. In \cite{13} it was proved that
$\exp(A)\in\P_{2n}(\exp(X))$. Then for any $m\in\IN$ \ $E_m=R(1-\frac
1m)\cap (P(X)\times\exp(A))$. Since the set $R(1-\frac 1m)\subset
P(X)\times\exp(X)$ is closed (see the proof of Theorem 1.22), $E_m\in\P_{2n}(P(X)\times\exp(X))$.
Consequently, $\pr_1(E_m)\in\P_{2n+1}(P(X))$ and
$\HP(A)=\bigcap_{m=1}^\infty \pr_1(E_m)\in\P_{2n+1}(P(X))$ \cite[\S38, III,
Theorem 3]{12}.
\end{proof}

\begin{remark}\label{r2.33} If $A$ is an analytic subset of the a compact space $X$, then $\HP(A)\in\P_4(P(X))$. This follows from Theorem \ref{t1.23} and the equality
$\HP(A)=\Pt(A)$, which is the result of the fact that analytic subsets of metric compacta are measurable with respect to any measure.
\end{remark}

\end{document}